\documentclass[a4paper,11pt]{article}

\usepackage{hyperref}
\usepackage{amsmath}
\usepackage{amsthm}
\usepackage{amsfonts}
\usepackage{amssymb}
\usepackage{verbatim}
\usepackage{tikz}
\usepackage[bottom,symbol]{footmisc}
\usepackage{geometry}
\usepackage{enumerate}
\usepackage[numbers]{natbib}
\title{Distribution of components in the $k$-nearest neighbour random geometric graph for $k$ below the connectivity threshold}

\author{Victor Falgas-Ravry  \thanks{Institutionen f\"or matematik och matematisk statistik, Ume{\aa}  Universitet, 901 87 Ume{\aa}, Sweden. Supported by a grant from the Kempe foundation. Part of this work was done while the author was an EPSRC-funded PhD student at Queen Mary, University of London. Email: {\tt victor.falgas-ravry@math.umu.se}.}}

\newtheorem{theorem}{Theorem}
\newtheorem{corollary}[theorem]{Corollary}
\newtheorem{lemma}[theorem]{Lemma}
\newtheorem{remark}{Remark}
\newtheorem{claim}{Claim}
\theoremstyle{definition}
\newtheorem{definition}{Definition}

\begin{document}
\maketitle
\begin{abstract}
Let $S_{n,k}$ denote the random geometric graph obtained by placing points inside a square of area $n$ according to a Poisson point process of intensity $1$ and joining each such point to the $k=k(n)$ points of the process nearest to it. In this paper we show that if $\mathbb{P}(S_{n,k} \textrm{ connected})>n^{-\gamma_1}$ then the probability that $S_{n,k}$ contains a pair of `small' components `close' to each other is $o(n^{-c_1})$ (in a precise sense of `small' and 'close'), for some absolute constants $\gamma_1>0$ and $c_1 >0$. This answers a question of Walters~\cite{Walters12}. (A similar result was independently obtained by Balister.)

As an application of our result, we show that the distribution of the connected components of $S_{n,k}$ below the connectivity threshold is asymptotically Poisson.
\end{abstract}

\section{Introduction}
In this paper, we study the distribution of small connected components below the connectivity threshold in the $k$-nearest neighbour model. We prove that they lie far apart, and are approximately distributed like a Poisson process in both a numerical and a spatial sense.

\subsection{Definitions}
Let $S_n$ denote the square $[0,\sqrt{n}]^2$ of area $n$, and let $||\cdot||$ denote the usual Euclidean norm in $\mathbb{R}^2$.

Let $k=k(n)$ be a nonnegative integer parameter. Scatter points at random inside $S_n$ according to a Poisson process of intensity $1$. (Or equivalently, scatter $N\sim \textrm{Poisson(n)}$ points uniformly at random inside $S_n$.) This gives us a random geometric vertex set $\mathcal{P}$.  The \emph{$k$-nearest neighbour graph} on $\mathcal{P}$, $S_{n,k}=S_{n,k}(\mathcal{P})$, is obtained by setting an undirected edge between every vertex of $\mathcal{P}$ and the $k$ vertices of $\mathcal{P}$ nearest to it.

The Poisson law of the random pointset $\mathcal{P}$ coupled with our deterministic definitions of $S_{n,k}$ gives rise to a probability measure on the space of $k$-nearest neighbour graphs with vertices in $S_n$, which we refer to as the \emph{$k$-nearest neighbour model} $\mathcal{S}_{n,k}$.

As this model has some inherent randomness, rather than proving deterministic statements about $S_{n,k}$ we are mostly interested in establishing that properties hold for `typical' $k$-nearest neighbour graphs. Formally, given a property $\mathcal{Q}$ of $k$-nearest neighbour graphs (i.e a sequence of subsets $\mathcal{Q}_n$ of the set of all geometric graphs on $S_n$) and a sequence $k=k(n)$ of nonnegative integers, we say that $S_{n,k(n)}$ has property $\mathcal{Q}$ \emph{with high probability} (\emph{whp}) if
\[\lim_{n \rightarrow \infty} \mathbb{P}(S_{n,k(n)} \in \mathcal{Q}_n)=1.\]

\subsection{Previous work on the k-nearest neighbour model}
An important motivation for the study of the connectivity properties of the $k$-nearest neighbour model comes from the theory of ad-hoc wireless networks: suppose we have some radio transmitters (nodes) spread out over a large area and wishing to communicate using multiple hops, and that each transmitter can adjust its range so as to ensure two-way radio contact with the $k$ nodes nearest to it. The connectivity of this radio network is modelled using $S_{n,k}$ in a natural way. As a result, the model has received a significant amount of attention, though many questions remain. (See e.g. the recent survey paper of Walters~\cite{Walters11}.)

Elementary arguments show that there exist constants $c_l \leq c_u$ such that for $k(n)\leq c_l \log n$ $S_{n,k}$ is whp not connected, while for $k(n) \geq c_u \log n$ $S_{n,k}$ is whp connected. A bound of $c_u \leq 5.1774$ was obtained by Xue and Kumar~\cite{XueKumar04}, using a substantial result of Penrose~\cite{Penrose97} for the related Gilbert disc model~\cite{Gilbert61}. An earlier bound of $c_u \leq 3.5897$ could also be read out of earlier work of Gonz\'ales--Barrios and Quiroz~\cite{GonzalesBarriosQuiroz03}. These results were substantially improved by Balister, Bollob\'as, Sarkar and Walters in a series of papers~\cite{BalisterBollobasSarkarWalters05, BalisterbollobasSarkarWalters09a, BalisterBollobasSarkarWalters09b} in which they established inter alia the existence of a critical constant $c_{\star}$: $0.3043< c_{\star}<0.5139$ such that for $c<c_{\star}$ and $k\leq c\log n$, $S_{n,k}$ is whp not connected while for $c>c_{\star}$ and $k\geq c\log n$, $S_{n,k}$ is whp connected. Building on their work, Walters and the author~\cite{FalgasRavryWalters12a} recently proved that the transition from whp not connected to whp connected is sharp in $k$: there is an absolute constant $C>0$ such that if $S_{n,k}$ is connected with probability at least $\varepsilon>0$ and $n$ is sufficiently large, then for $k'\geq k+ C\log(1/\varepsilon)$, $S_{n,k'}$ is connected with probability at least $1-\varepsilon$.

As part of their results, Balister, Bollob\'as, Sarkar and Walters~\cite{BalisterBollobasSarkarWalters05} showed that when $0.3\log n<k<0.6 \log n$ whp all of the following hold:
\begin{enumerate}[(i)]
\item all edges have length $O(\sqrt{\log n})$,
\item there is a unique `giant' connected component,
\item all other components have diameter $O(\sqrt{\log n})$.  
\end{enumerate}
(Strictly speaking, they only showed that $S_{n,k}$ has at most one giant component; that such a component exists follows from the study of percolation in the $k$-nearest neighbour graph: see e.g.~\cite{BalisterBollobas12}.)

Refining their techniques, Walters~\cite{Walters12} showed that around the connectivity threshold, there are no `small' components (of diameter $O(\sqrt{\log n})$) lying `close' to the boundary of $S_n$ (within distance $O(\log n)$), and used this to improve the upper bound on $c_{\star}$ to $0.4125$.
Towards the end of his paper~\cite{Walters12}, Walters asked a number of questions about the properties of small components of $S_{n,k}$ below the connectivity threshold, with the aim of completing the picture we currently have and perhaps improving the upper bound on $c_{\star}$.

\subsection{Results of the paper}
In this paper we contribute to this project by proving that the `small' components are whp far apart (or more precisely, not `close' together) and that they are distributed like a Poisson point process inside $S_n$. This answers one of Walters's questions.

To state our results formally, we must ascribe a precise meaning to `small' and `close'. We recall a result of Balister, Bollob\'as, Sarkar and Walters to do this.
\begin{lemma}[Lemma~1 of \cite{BalisterBollobasSarkarWalters09b}]
For any fixed $\alpha_1, \alpha_2$ with $0<\alpha_1<\alpha_2$ and any $\beta>0$, there exists $\lambda=\lambda(\alpha_1,\alpha_2,\beta)>0$, depending only on $\alpha_1,\alpha_2$ and $\beta$, such that for any $k$ with $\alpha_1\log n\leq k \leq \alpha_2 \log n$, the probability that $S_{n,k}$ contains two components each of diameter at least $\lambda\sqrt{\log n}$ or any edge of length at least $\lambda\sqrt{\log n}$ is $O(n^{-\beta})$.
\end{lemma}
\begin{definition}\label{mconbis}
Let $\lambda= \max(\lambda(0.3, 0.6, 2), e^2 )$. A component shall be deemed \emph{small} if it has diameter less than $\lambda \sqrt{\log n}$. Two small components shall be deemed \emph{close} if they contain two points lying at a distance less than $8\lambda\sqrt{\log n}$ from one another. 
\end{definition}

Our main result is the following:
\begin{theorem}\label{nosmallclosefinal}
There exist absolute constants $\gamma_1>0$ and $c_1>0$  such that if 
\[\mathbb{P}(S_{n,k} \textrm{ connected})> n^{-\gamma_1}\] 
then
\[\mathbb{P}(S_{n,k} \textrm{ contains a pair of small, close components})= o(n^{-c_1}).\]
\end{theorem}
This answers Question~1  in Walters~\cite{Walters12}. Balister has independently obtained a similar result (personal communication).
\begin{remark}
Theorem~1 of~\cite{FalgasRavryWalters12a}
 showed that, for $n$ large enough, we need to increase $k$ by at most a constant times $\log (1/\varepsilon)$ for $S_{n,k}$ to go from having an $\varepsilon$ chance of being connected to having a $1-\varepsilon$ chance of being connected.
Assuming there is a `bluntness' converse to this result, i.e. that we need to increase $k$ by at least a (smaller) constant times $\log (1/\varepsilon)$ for this transition to occur, then there must be some constant $\delta: \ 0 < \delta < c_{\star}$ such that for all $k=k(n)$ with $k(n)>(c_{\star}- \delta)\log n$ we have $\mathbb{P}(S_{n,k} \textrm{ connected})>n^{-\gamma_1}$.  In particular Theorem~\ref{nosmallclosefinal} is not vacuous since there is a range of $k$ for which 
\[n^{-\gamma_1} < \mathbb{P}(S_{n,k} \textrm{ connected}) << 1-o(n^{-c_1}).\]
(For example, $k= \lfloor c \log n \rfloor$ for any $c$ with $c_{\star}- \delta < c < c_{\star}$ will do.)
\end{remark}

\noindent Next we turn to the distribution of the small components of $S_{n,k}$. Let $X=X_{n,k}$ denote the number of small connected components of $S_{n,k}$. (Since there is whp a unique non-small connected component~\cite{BalisterBollobasSarkarWalters05}, $X$ is whp the number of components of $S_{n,k}$ minus $1$.)
Also, given $\nu\geq 0$ and $A \subseteq \mathbb{N}\cup\{0\}$, let $\textrm{Po}_{\nu}(A)$ denote the probability a Poisson random variable with parameter $\nu$ takes a value inside $A$.

As an application of Theorem~\ref{nosmallclosefinal}, we prove:
\begin{theorem}\label{poissondistrfinal} There exist absolute constants $\gamma_2$ and $c_2>0$ such that if $k=k(n)$ is an integer sequence with $\mathbb{P}(S_{n,k} \textrm{ connected}) > n^{-\gamma_2}$
for all $n$, then, writing $\nu=\nu(n)$ for $- \log \left( \mathbb{P}(S_{n,k} \textrm{ connected}) \right)$, we have
\[\sup_{A \subseteq \mathbb{N}\cup\{0\}} \left \vert \mathbb{P}(X\in A) - \textrm{Po}_{\nu}(A) \right \vert= o(n^{-c_2}).\]
\end{theorem}

\begin{corollary}
Let $k(n)$ be an integer sequence. Suppose there is a subsequence $\left( k(n_i)\right)_{i \in \mathbb{N}}$ such that
\[\mathbb{P}(S_{n_i,k(n_i)} \textrm{\  connected}) \rightarrow e^{-\nu}\]
for some constant $\nu\geq 0$. Then the law of $X_{n_i, k(n_i)}$ converges in distribution to Poisson with parameter $\nu$: 
\[ \mathcal{L}(X_{n_i, k(n_i)}) \xrightarrow{d}\rm{Poisson}(\nu).\]
\end{corollary}

We also prove a spatial analogue of Theorem~\ref{poissondistrfinal}: not only is the number of small components (approximately) Poisson distributed, but their spatial location is (approximately) distributed according to a Poisson point process inside $S_n$. We defer a precise statement of this result, Theorem~\ref{poissondistrspace}, until Section~4.

\subsection{Structure of the paper}
We follow the strategy introduced by Balister, Bollob\'as, Sarkar and Walters in~\cite{BalisterBollobasSarkarWalters09b} and developed in~\cite{FalgasRavryWalters12a}: we prove Theorem~\ref{nosmallclosefinal} by looking at local events.  In Section 2, we prove a local result, Theorem~\ref{localnoclose}, which can be thought of as an analogue of the local sharpness result Lemma~12 in~\cite{FalgasRavryWalters12a}.

Theorem~\ref{localnoclose} is then used in Section~3 together with local-global correspondence results from~\cite{FalgasRavryWalters12a} to prove the global result Theorem~\ref{nosmallclosefinal}.

Finally in the last section we use a form of the Chen--Stein Method~\cite{Chen75, Stein72} due to Arratia, Goldstein and Gordon~\cite{ArratiaGoldsteinGordon89} together with our results from the first two sections to prove Theorems~\ref{poissondistrfinal} and~\ref{poissondistrspace} on the distribution of small components.

\section{Proof of the local theorem}
In this section we consider the connectivity of the $k$-nearest neighbour random geometric graph model on a local scale (i.e. within a region of area $O(\log n)$).

Pick $n$ sufficiently large. Let $M= \max(160 \lceil\lambda\rceil, 50)$. (We remark that this is similar to but slightly larger than the choice of $M$ in~\cite{FalgasRavryWalters12a}.) We consider a Poisson point process of intensity $1$ in the box 
\[U_n= {[-\frac{M}{2}\sqrt{\log n}, +\frac{M}{2}\sqrt{\log n}]}^2.\] 
We place an undirected edge between every point and its $k$ nearest neighbours to obtain the graph $U_{n,k}$.

Let us define two families of events related to the connectivity of $U_{n,k}$: 
\begin{definition}\label{akbkdef}
Let $A_k$ be the event that $U_{n,k}$ has a connected component wholly contained inside the central subsquare $\frac{1}{2}U_n$. Let $B_k$ be the event that $U_{n,k}$ has at least two connected components wholly contained inside the central subsquare $\frac{1}{2}U_n$. 
\end{definition}

Our aim is to prove:
\begin{theorem} \label{localnoclose}
There exist constants $c_3,c_4>0$ such that for all integers $k$ with $k \in (0.3\log n, 0.6 \log n)$, we have
\[\mathbb{P}(B_k)\leq c_3n^{-c_4}\mathbb{P}(A_k) +o(n^{-2}). \]
\end{theorem}
This is saying that on a local scale it is far less likely that we have two small connected components close together than just one small connected component on its own. Our proof strategy is as follows: we show that whp if $B_k$ occurs then there must be a large empty region inside $U_n$ to which many points can be added without joining up all the components of $U_{n,k}$ which are contained inside $\frac{1}{2}U_n$. This ensures that $A_k$ still occurs for the new pointset, and can be exploited to show $A_k$ is much more likely than $B_k$.

To prove Theorem~\ref{localnoclose}, we shall need a simple lemma on the concentration of Poisson random variables.
\begin{lemma}\label{poisson bound}
There exist constants $\lambda_1$ and $\lambda_2$, such that for all integers $k$ with $k$ between $0.3\log n$ and $0.6 \log n$, the probability that there is any point $x\in U_{n}$(not necessarily in the pointset arising from the Poisson point process) such that the ball of radius $\lambda_1 \sqrt{\log n}$ about $x$ contains at least $k$ vertices of $U_{n,k}$ or that the ball of radius $\lambda_2 \sqrt{\log n}$ about $x$ contains fewer than $k$ vertices of $U_{n,k}$ is $o(n^{-2})$.

Moreover, $\lambda_2$ can be chosen to be less than $\lambda$.
\end{lemma}
\begin{remark}
This says that the probability $U_{n,k}$ contains a pair of vertices not joined by an edge and lying within distance $\lambda_1\sqrt{\log n}$ of each other, or a pair of vertices joined by an edge and lying at distance at least $\lambda_2 \sqrt{\log n}$ from one another is $o(n^{-2})$, i.e. a negligible quantity.
\end{remark}

\begin{figure}
\begin{center}
\begin{tikzpicture}
\fill[fill=green!10, draw=green!100] (-2.5, 2.5)-- (-0.5, 2.5)  arc (0:-90:2cm) -- (-2.5, 2.5);
\node (x) at ( -2.5,2.5) [inner sep=0.5mm, circle, fill=black!100] {};
\node [red, above] at (x.west) {x};
\draw[step=0.5, xshift=0.5cm, yshift=0.5cm] (-3.5,-1.5) grid (1, 3);
\node [red] at(-1.25, 1.75) {D};
\node [black, above] at (-1, 3.5) {$U_n$}; 

\fill[fill=green!10, draw=green!100] 
(5.5,1) circle (1cm);

\node (x) at ( 5.5,1) [inner sep=0.5mm, circle, fill=black!100] {};
\node [red, above] at (x.west) {x};
\draw[step=0.5, xshift=0.5cm, yshift=0.5cm] (3.5,-1.5) grid (7.5, 3);
\draw (4, -1) -- (4, 3.5);
\node [red] at (6.25, 1.25) {D};
\node [black, above] at (5.5, 3.5) {$U_n$}; 
\end{tikzpicture}
\end{center}
\caption{The square grid $\Gamma$, the point $x$ and the quarter-disc and complete disc considered in the proof of Lemma~\ref{poisson bound}.}
\end{figure}
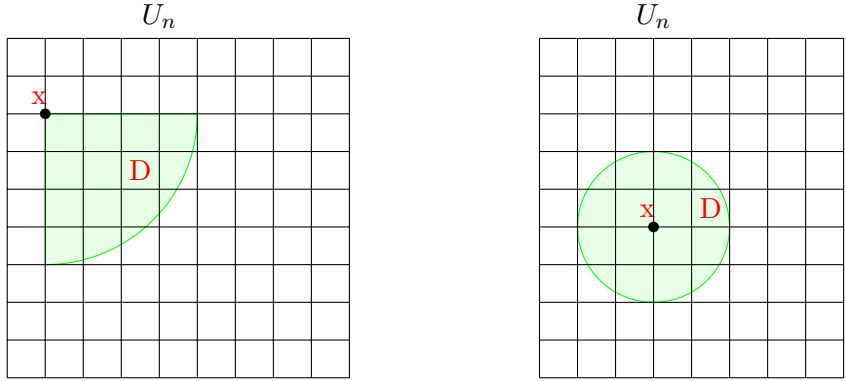

\begin{proof}[Proof of Lemma~\ref{poisson bound}]
We first show that there is a constant $\lambda_2\leq \lambda$ such that the probability that there is some point in $U_n$ with fewer than $0.6\log n$ vertices of $U_{n,k}$ within distance $\lambda_2 \sqrt{\log n}$ of itself is $o(n^{-2})$.

Let $\Gamma$ denote the intersection of the integer grid $\mathbb{Z}^2$ with $U_n$.  Let $\lambda_2'= 2\sqrt{\frac{e^3}{\pi}}$ and consider any $x \in \Gamma$. Since $M>50$, at least one of the four standard quadrant quarter discs of radius $\lambda_2' \sqrt{\log n}$ about $x$ is contained inside $U_n$; call this quarter disc $D$. The probability that $D$ contains fewer than $0.6 \log n$ vertices of $U_{n,k}$ is at most
\begin{align*}
&0.6 (\log n)\exp\left(- \frac{\pi \lambda_2'^2}{4}\log n\right)\frac{(\pi \lambda_2'^2\log n/4)^{0.6\log n}}{(0.6\log n) !}\\
&< 0.6 (\log n)\exp\left( -\frac{\pi \lambda_2'^2}{4}(\log n)+0.6(\log n) \log (\frac{e \pi \lambda_2'^2}{0.6\times4})\right)\\
&< 0.6 (\log n)\exp\left( -(e^3 - 3)\log n\right)\\
&< 0.6 (\log n) n^{-4}.
\end{align*}
Thus the probability that $\Gamma$ contains any  point with fewer than $0.6 \log n$ vertices of $U_{n,k}$ within distance $\lambda_2' \sqrt{\log n}$ is less than $\vert \Gamma \vert \times \frac{0.6\log n} {n^{4}}=o(n^{-2})$. Since every point in $U_n$ lies at distance at most $\sqrt{2}$ from a point in $\Gamma$, this proves our claim for $\lambda_2=\lambda_2' +1$, say. This is at most $e^2$, which is less than $\lambda$, as required.

Now let us show that there is a constant $\lambda_1$ such that the probability that there is some point in $U_n$ having more than $0.3\log n$ vertices of $U_{n,k}$ within distance $\lambda_1 \sqrt{\log n}$ of itself is $o(n^{-2})$.

Let $\Gamma$ again denote the intersection of the integer grid $\mathbb{Z}^2$ with $U_n$.  Let $\lambda_1'=\sqrt{\frac{\exp(-49/3)}{\pi}}$ and consider any $x \in \Gamma$. Let $D$ denote the intersection of the disc of radius $\lambda_1'\sqrt{\log n}$ about $x$ with $U_n$. The probability that $D$ contains more than $0.3 \log n$ vertices of $U_{n,k}$ is at most
\begin{align*}
&\frac{{\vert D \vert}^{\lceil 0.3 \log n \rceil}}{\lceil 0.3 \log n \rceil !}\exp\left(- \vert D \vert\right)+ \frac{{\vert D \vert}^{\lceil 0.3 \log n \rceil+1}}{(\lceil 0.3 \log n \rceil +1)!}\exp\left(- \vert D \vert\right)+ \cdots\\
&< \frac{{(\pi \lambda_1'^2 \log n)}^{0.3\log n}}{\lceil 0.3 \log n \rceil !} e^{- \pi \lambda_1'^2\log n}\left(1 + \frac{\pi \lambda_1'^2}{0.3}+{\left(\frac{\pi \lambda_1'^2}{0.3}\right)}^2+ \cdots \right)\\
&<\exp\left( 0.3 \log n \left(\log \left(\pi \lambda_1'^2\log n\right) - \log \left(0.3\log n/e\right)- \frac{\pi \lambda_1'^2}{0.3}\right)\right) \frac{1}{1- \pi \lambda_1'^2/0.3}\\
&<\exp\left( 0.3 \log n \left( -\frac{49}{3}+\log \frac{e}{0.3} \right)+O(1)\right)\\
&<\exp\left(-4 \log n +O(1)\right)\\
&=o(n^{-3}).
\end{align*}
Thus the probability that $\Gamma$ contains any point with more than $0.3 \log n$ vertices of $U_{n,k}$ within distance $\lambda_1' \sqrt{\log n}$ is less than $\vert\Gamma\vert \times \exp\left(-4\log n +O(1)\right)=o(n^{-2})$. Now every point in $U_n$ lies at distance at most $\sqrt{2}$ from a point in $\Gamma$, so this proves our claim for $\lambda_1=\lambda_1'/2$, say.
\end{proof}

\begin{proof}[Proof of Theorem~\ref{localnoclose}]
 
Let us assume that there is no point in $U_n$ with more than $k$ vertices of $U_{n,k}$ within distance $\lambda_1 \sqrt{\log n}$ of itself or fewer than $k$ vertices of $U_{n,k}$ within distance $\lambda_2 \sqrt{\log n}$ of itself. We shall denote by $\mathcal{C}$ the set of pointsets we are thus excluding. By Lemma~\ref{poisson bound}, $\mathbb{P}(\mathcal{C})=o(n^{-2})$.

We consider a perfect tiling of $U_n$ into tiles of area $\frac{\log n}{N^2}$, for some (large) constant $N$. Explicitly, we shall choose
\[N= \max(N_1, N_2, N_3),\]
where $N_1$, $N_2$ and $N_3$ are the constants
\begin{align*}
N_1&= \left\lceil \sqrt{5}/\lambda_1\right\rceil+1,\\
N_2&= \left\lceil \frac{2}{\lambda_1}+ \frac{4\sqrt{5} \lambda_2}{{\lambda_1}^2}\right\rceil, \textrm{ and} \\
N_3&= \left \lceil \frac{1}{{\lambda_1}^2}\left((1+ \sqrt{5})\lambda_1+ \lambda_2+\sqrt{{((1+\sqrt{5})\lambda_1 + \lambda_2)}^2- (5+2\sqrt{5}){\lambda_1}^2}\right)\right \rceil+1,\\
\end{align*}
each of which will appear at one of the stages of our argument. The choice of $N\geq N_2$ ensures that the inequality 
\[  \frac{1}{N} + \left(\frac{4\sqrt{5} \lambda_2}{N}+ \frac{1}{N^2} \right)^{1/2} \leq \lambda_1\]
holds, while the choice of $N\geq N_3$ ensures that 
\[ \frac{1}{N^2}+ \frac{2\lambda_2}{N} < \left( \lambda_1 - \frac{(1+\sqrt{5})}{N}\right)^2\]
holds. (Both inequalities clearly hold for all sufficiently large $N$, and it is easily checked by solving two quadratic equations that the values of $N_2$ and $N_3$ given above will do.)

Given a pointset $\mathcal{P}\subset U_n$, write $U_{n,k}(\mathcal{P})$ for the $k$-nearest neighbour graph on $\mathcal{P}$.
\begin{definition}\label{defbkq}
For each tile $Q$, let $B_k(Q)$ be the event that the pointset $\mathcal{P}$ resulting from the Poisson process on $U_n$ has the following properties:
\begin{enumerate}[(i)]
\item $\mathcal{P} \in B_k$ (i.e. $U_{n,k}(\mathcal{P})$ has at least two connected components contained inside $\frac{1}{2}U_n$),
\item $Q$ contains no point of $\mathcal{P}$, and
\item for any set of points $\mathcal{B}\subset Q$, the pointset $\mathcal{P}\cup\mathcal{B}$ lies in $A_k$ (i.e. $U_{n,k}(\mathcal{P}\cup\mathcal{B})$ has at least one connected component wholly contained inside $\frac{1}{2}U_n$).
\end{enumerate}
\end{definition}

The key step in the proof of Theorem~\ref{localnoclose} is to show:
\begin{theorem}\label{bktileunion}
\[B_k \setminus \mathcal{C} \subseteq \bigcup_Q B_k(Q).\]
\end{theorem}
In other words, if $B_k$ occurs we can (except in a negligible proportion of cases) find an empty tile to which we can add many points and still have $A_k$ occurring in the resulting pointset. This can be thought of as an analogue of the key Lemma~7 in~\cite{FalgasRavryWalters12a}.
\begin{proof}[Proof of Theorem~\ref{bktileunion}]
Let $\mathcal{P}$ be a pointset for which $B_k\setminus \mathcal{C}$ occurs. Say that a tile is \emph{empty} if it contains no point of $\mathcal{P}$. Let $X, Y$ be the vertex sets of two small connected components of $U_{n,k}(\mathcal{P})$ witnessing $B_k$. (At least two such components must exist, though there could potentially be more.)

Let $a$ be a vertex in $X\cup Y$ nearest to the bottom side of $U_n$. Without loss of generality, we may assume that $a \in X$. Let $E$ be the horizontal line through $a$. Then all points of $X$ and $Y$ must lie on or above $E$. Now $a$ lies in some tile, $Q_a$ say.

We consider the tiles directly below $Q_a$. Since $N\geq N_1> \frac{\sqrt{5}}{\lambda_1}$, the topmost of these tiles must be empty. There are two cases to consider. Either all tiles directly below $Q_a$ are empty, in which case we let $Q$ denote the one among them which is incident with the boundary of $U_n$; or there is some tile directly below $Q_a$, which is nonempty. Then let $Q'$ be the topmost of these nonempty tiles, and let $Q$ denote the (empty) tile directly above it.

\begin{lemma}
\[\mathcal{P}\in B_k(Q).\]
\end{lemma}
\begin{proof}
Let $\mathcal{B}\subset Q$ be a nonempty set of points in $Q$, and let $\mathcal{P}'=\mathcal{P}\cup \mathcal{B}$. Our claim is that $\mathcal{P}' \in A_k$. To establish this, it is enough to show that there are no edges from $Y$ to $Y^c$ in $U_{n,k}(\mathcal{P}')$ (as then $Y$ will be a connected component of $U_{n,k}(\mathcal{P}')$ contained inside $\frac{1}{2}U_n$).

Since the only edges in $U_{n,k}(\mathcal{P}')$ that are not also edges of $U_{n,k}(\mathcal{P})$ have at least one end in $Q$, it suffices to show no vertex of $Y$ is joined by an edge of $U_{n,k}(\mathcal{P}')$ to a point $b \in \mathcal{B}$. We split into two cases.

First of all, suppose $Q$ is incident with the boundary of $U_n$. Since $\mathcal{P}\notin \mathcal{C}$ and $\lambda_2<\frac{M}{4}$ (since $\lambda_2\leq \lambda\leq\frac{M}{160}$), we know that no point in $Q$ can have any of its $k$ nearest neighbours inside $\frac{1}{2}U_n$ and vice-versa. Thus there are no edges between $Q$ and $Y\subseteq \frac{1}{2}U_n$ in $U_{n,k}(\mathcal{P}')$, and $\mathcal{P}' \in A_k$ as required.

We now turn to the less trivial case where $Q$ is not incident with the boundary of $U_n$. Then the tile $Q'$ directly below $Q$ is nonempty: there exists $c \in \mathcal{P}\cap Q'$.

Let $R$ denote the distance between $c$ and its $k$\textsuperscript{th} nearest neighbour in $\mathcal{P}$. For any $b \in \mathcal{B}$, $|| c -b || \leq \frac{\sqrt{5}}{N} \sqrt{\log n}$. Thus the distance between $b$ and its $k$\textsuperscript{th} nearest neighbour in $\mathcal{P}'$ will be at most $R+ \frac{\sqrt{5}}{N}\sqrt{\log n}$. Also, $ac$ is not an edge of the $k$-nearest neighbour graph on $\mathcal{P}$, whence $|| a -c || \geq R$ and $|| a-b|| \geq R-\frac{\sqrt{5}}{N}\sqrt{\log n}$ for all $b \in \mathcal{B}$.

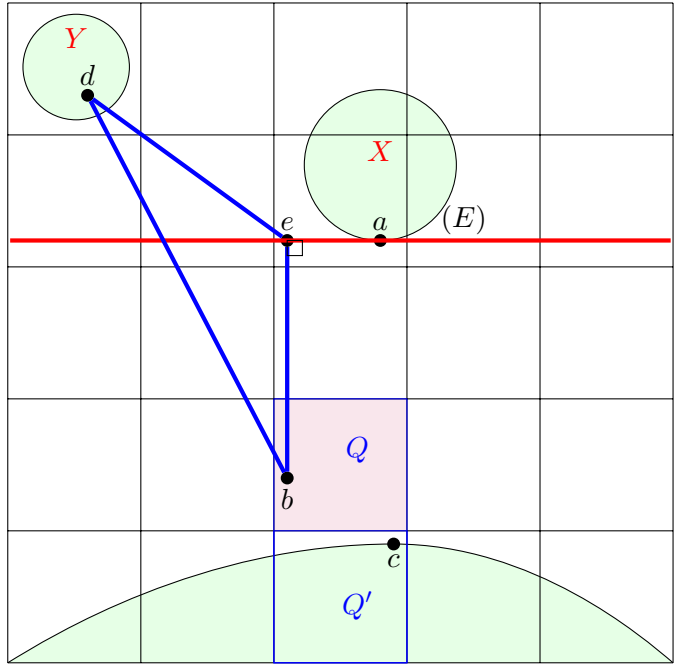
\begin{figure}
\begin{center}
\begin{tikzpicture}[inner sep=0.5mm, sommet/.style={circle,draw=black!100,fill=black!100,thick},lg/.style={inner sep=0mm, circle, thick}]

\fill[fill=green!10, draw=black!100] (4.9, 6.6) circle (1);
\fill[fill=green!10, draw=black!100] (0.9, 7.9) circle (0.7);
\node[red, above] at (4.9, 6.6) {$X$};
\node[red, above] at (0.9, 8.1) {$Y$};
\fill[purple!10] (3.5, 1.75) rectangle (5.25, 3.5);
\node[blue, above] at (4.6, 2.6) {$Q$};

\fill[fill=green!10, draw=black!100]
(0,0) parabola bend (5.075,1.575) (8.75, 0);
\node[blue, below] at (4.6, 1) {$Q'$};

\node[sommet] (anode) at (4.9, 5.6) [label=above:$a$] {};
\node[sommet] (bnode) at (3.675, 2.45) [label=below:$b$] {};
\node[sommet] (cnode) at (5.075, 1.575) [label=below:$c$] {};
\node[sommet] (dnode) at (1.05, 7.525) [label=above:$d$] {};
\node[sommet] (enode) at (3.675, 5.6) [label=above:$e$] {};

\node[lg] (leftline) at (0, 5.6) {};
\node[lg] (midline) at (6, 5.6) [label=above:$(E)$]{};
\node[lg] (rightline) at (8.75, 5.6) {};
\draw[step=1.75cm]  (0 ,0) grid (8.75,8.75);
\draw (leftline) to (rightline) [red, ultra thick];

\draw (bnode) to (enode) [blue, ultra thick];
\draw (enode) to (dnode) [blue, ultra thick];
\draw (bnode) to (dnode) [blue, ultra thick];
\draw (3.675, 5.6) rectangle (3.875, 5.4);

\draw[blue] (3.5, 0) rectangle (5.25, 3.5);
\draw[blue] (3.5, 0) rectangle (5.25, 1.75);
\end{tikzpicture}
\end{center}
\caption{The points $a,b,c,d$ and $e$, the line $E$, the tile $Q$ and the components $X$ and $Y$.}
\end{figure}

Now, let $b \in \mathcal{B}\subseteq Q$ and suppose $d$ is a point lying above $E$ such that $d$ is one of the $k$ nearest neighbours of $b$ in $\mathcal{P}'$. Our earlier choice of $N\geq N_2$ ensures the following holds:
\begin{claim}\label{distance}
\[|| a - d || \leq \lambda_1 \sqrt{\log n}.\]
\end{claim}
\begin{remark}
By our assumption on $\mathcal{P}$ (namely our assumption that $\mathcal{P} \notin \mathcal{C}$), this implies that $ad$ is an edge in $U_{n,k}(\mathcal{P})$. Since $a\in X$ it follows that $d \notin Y$.
\end{remark}
\begin{proof}[Proof of Claim~\ref{distance}]
This is an exercise in Euclidean geometry. (See Figure~2.)
We know that $|| b -d || \leq R + \frac{\sqrt{5}}{N} \sqrt{\log n}$. Let $e$ be the foot of the perpendicular to $E$ which goes through $b$. As $b$ lies in a tile directly below $a$'s tile, we have $|| a - e|| \leq \frac{1}{N} \sqrt{\log n}$. It follows by Pythagoras's Theorem that
\begin{align*}
{|| b -e ||}^2&={|| b -a||}^2-{|| a-e||}^2\\
&\geq (R-\frac{\sqrt{5}}{N}\sqrt{\log n})^2- \frac{1}{N^2}\log n.
\end{align*}
Now the angle $bed$ is obtuse. Hence,
\begin{align*}
{|| e-d||}^2 & \leq {|| b -d ||}^2-{|| b-e||}^2\\
&\leq (R+\frac{\sqrt{5}}{N}\sqrt{\log n})^2-(R-\frac{\sqrt{5}}{N}\sqrt{\log n})^2+\frac{1}{N^2}\log n\\
&= \frac{4\sqrt{5}}{N} R\sqrt{\log n} + \frac{1}{N^2}\log n.
\end{align*}
Finally, we have
\begin{align*}
|| a -d || &\leq || a-e||+|| e-d||\\
&\leq \frac{1}{N}\sqrt{\log n}+\left(\frac{4\sqrt{5}}{N} R\sqrt{\log n} + \frac{1}{N^2}\log n \right)^{\frac{1}{2}}.
\end{align*}

Now $R\leq \lambda_2 \sqrt{\log n}$ (since $\mathcal{P}\notin \mathcal{C}$). Substituting this into the above yields
\begin{align*}
|| a-d|| & \leq \left(\frac{1}{N} +\left( \frac{4\sqrt{5}\lambda_2}{N}+\frac{1}{N^2}\right)^{1/2} \right)\sqrt{\log n},
\end{align*}
which by our choice of $N\geq N_2$ is less than $\lambda_1 \sqrt{\log n}$.
\end{proof}

On the other hand, suppose $d$ is a point lying above the line $E$ such that $b$ is one of the $k$ nearest neighbours of $d$ in $\mathcal{P}'$. Then our earlier choice of $N\geq N_3 $ ensures the following holds:

\begin{claim}\label{distance2}
\[|| a-d||< || b-d ||.\]
\end{claim}
\begin{remark}
This implies that $a$ is one of the $k$ nearest neighbours of $d$ in $U_{n,k}(\mathcal{P}')$, and hence also in $U_{n,k}(\mathcal{P})$. As $a\in X$ it follows that $d \notin Y$.
\end{remark}
\begin{proof}[Proof of Claim~\ref{distance2}]
This is again an exercise in Euclidean geometry. (See Figure~2.)
Since $b$ is among the $k$ nearest neighbours of $d$, it follows that $|| b-d|| \leq \lambda_2 \sqrt{\log n}$ (since $d\in\mathcal{P}$ and $\mathcal{P}\notin \mathcal{C}$). Similarly, as $ac$ is not an edge of $U_{n,k}(\mathcal{P})$, we must have that $|| a-c|| \geq \lambda_1 \sqrt{\log n}$. Let $e$ be the foot of the perpendicular to the line $E$ which goes through $b$. Since the triangle $bed$ is obtuse, we have $|| e-d|| \leq || b-d||$. As $b$ lies in a tile directly below $a$'s tile, we have $|| a-e|| \leq \frac{1}{N} \sqrt{\log n}$. Also, as $c$ lies in a tile directly below $b$'s tile we have $|| b-c|| \leq\frac{\sqrt{5}}{N}\sqrt{\log n}$. Thus
\begin{align*}
|| b- e|| & \geq || a-c||- || a-e||-|| b-c||\\
& \geq \left(\lambda_1 - \frac{(1+\sqrt{5})}{N}\right)\sqrt{\log n}.
\end{align*}

Using again the fact that the triangle $bed$ is obtuse, we have
\begin{align*}
{|| b-d||}^2 & \geq {|| b-e||}^2+{|| e-d||}^2.\\
& \geq {\left(\lambda_1- \frac{(1+\sqrt{5})}{N}\right)}^2 \log n +{|| e-d||}^2. \tag{1}
\end{align*}
Finally we have
\begin{align*}
{|| a-d||}^2 &\leq \left(|| a-e||+|| e-d||\right)^2\\
& = {|| a-e||}^2+2|| a-e|| \times || e-d||+ {|| e-d||}^2\\
& \leq \frac{\log n}{N^2}+ \frac{2\sqrt{\log n}}{N}|| e-d||+{|| e-d||}^2\\
&\leq \frac{\log n}{N^2}+ \frac{2\lambda_2 \log n}{N}+{|| e-d||}^2,   \tag{2} 
\end{align*}
where the last line follows from the fact that $|| e-d|| \leq || b-d|| \leq \lambda_2 \sqrt{\log n}$.

Now our choice of $N$ (more specifically our choice of $N\geq N_3$) guarantees that
\[\frac{1}{N^2} + \frac{2\lambda_2}{N}  < (\lambda_1- \frac{(1+\sqrt{5})}{N})^2\]
so that by comparing (1) and (2) we have
\[|| a-d|| < || b-d||\]
as claimed.
\end{proof}

As remarked Claim~\ref{distance} tells us that if $d$ is one of $b$'s $k$ nearest neighbours in $\mathcal{P}'$ then $d$ is one of $a$'s $k$ nearest neighbours in $\mathcal{P}$ (since $\mathcal{P} \notin \mathcal{C}$) --- and in particular $d \notin Y$. On the other hand Claim~\ref{distance2} tells us that if $b$ is one of $d$'s $k$ nearest neighbours in $\mathcal{P}'$ then $a$ is one of $d$'s $k$ nearest neighbours in $\mathcal{P}$ --- so that again $d \notin Y$.

Combining these two claims, we see that there are no edges between $\mathcal{B}$ and $Y$ in $U_{n,k}(\mathcal{P}')$, and hence that $Y$ is a connected component of $U_{n,k}(\mathcal{P}')$ contained inside $\frac{1}{2}U_n$. We thus have $\mathcal{P}'\in A_k$, as claimed.
\end{proof}

For every $\mathcal{P} \in B_k \setminus \mathcal{C}$, we have thus shown there exists a tile $Q$ such that $\mathcal{P} \in B_k(Q)$, proving Theorem~\ref{bktileunion}.
\end{proof}

Having established Theorem~\ref{bktileunion}, the rest of the proof of Theorem~\ref{localnoclose} is straightforward. We can consider a Poisson process of intensity $1$ on $U_n$ as the union of two independent Poisson processes on $U_n \setminus Q$ and $Q$ respectively. Call the corresponding random pointsets $\mathcal{P}$ and $\mathcal{B}$ respectively. The event $B_k(Q)$ can then be considered as a product event
\[B_k(Q)= \{\mathcal{P} \in \tilde{B}_{k}(Q)\}\times \{\mathcal{B}=\emptyset\},\]
where $\tilde{B}_{k}(Q)$ is an event depending only on the points inside $U_n \setminus Q$.

We then have
\begin{align*}
\mathbb{P}(B_k(Q))&=\mathbb{P}(\tilde{B}_k(Q)) \mathbb{P}(\mathcal{B} =\emptyset)\\
& =\mathbb{P}(\tilde{B}_k(Q))\exp\left(-\frac{\log n}{N^2}\right)\\
&\leq \mathbb{P}(A_k)\exp\left(-\frac{\log n}{N^2}\right) \tag{3}
\end{align*}
where the last line follows from property (iii) of $B_k(Q)$ (which stated that if $B_k(Q)$ occurs then no matter what points we add to $Q$, the event $A_k$ will occur in the resulting modified pointset).

Now 
\begin{align*}
\mathbb{P}(B_k)&\leq \mathbb{P}(B_k\setminus \mathcal{C})+ \mathbb{P}(\mathcal{C})\\
&= \mathbb{P}\left(\bigcup_Q B_k(Q)\right) +\mathbb{P}(\mathcal{C}) &&\textrm{by Theorem~\ref{bktileunion}}\\
&\leq \sum_Q \mathbb{P}(B_k(Q)) + \mathbb{P}(\mathcal{C})\\
&\leq (MN)^2 \mathbb{P}(A_k)\exp\left(-\frac{\log n}{N^2}\right) + \mathbb{P}(\mathcal{C})&&\textrm{by (3)}\\
&=(MN)^2n^{-1/N^2}\mathbb{P}(A_k) +o(n^{-2})&&\textrm{by Lemma~\ref{poisson bound}}
\end{align*}
This concludes the proof of Theorem~\ref{localnoclose} with $c_3 =(MN)^2$ and $c_4=\frac{1}{N^2}$.
\end{proof}

\section{Proof of the global theorem}
In this section we prove our global result, Theorem~\ref{nosmallclosefinal}. It will be convenient to prove something slightly stronger than Theorem~\ref{nosmallclosefinal} (albeit with a more cumbersome statement), namely:
\begin{theorem}\label{nosmallclose}
There exist constants $\gamma_1'$ and $c_1'>0$ such that for every $\varepsilon:\ 0 < \varepsilon \leq \frac{1}{2}$, all $n>\varepsilon^{-1/\gamma_1'}$ and all integers $k\in (0.3\log n, 0.6 \log n)$, if
\[\mathbb{P}(S_{n,k} \textrm{ connected})\geq \varepsilon \]
holds, then
\[\mathbb{P}(S_{n,k}\textrm{ contains a pair of small, close components})<\left(\log \frac{ 1}{\varepsilon}\right)n^{-c_1'}.\]
\end{theorem}
\begin{proof}[Proof of Theorem~\ref{nosmallclosefinal} from Theorem~\ref{nosmallclose}]
The proof of Theorem~5 of~\cite{BalisterBollobasSarkarWalters05} establishes that there exists a constant $\gamma>0$ such that for any $k\leq 0.3 \log n$, the probability that $S_{n, k}$ is connected is $o(n^{-\gamma})$. Also, Theorem~13 of~\cite{BalisterBollobasSarkarWalters05} shows that for $k\geq 0.6\log n$, the probability that $S_{n,k}$ contains any small component is $o(n^{-(0.6\log 7 -1)})$. Thus Theorem~\ref{nosmallclosefinal} is immediate from Theorem~\ref{nosmallclose} together with an appropriate choice of the constants $\gamma_1$ and $c_1$.
\end{proof}

\begin{proof}[Proof of Theorem~\ref{nosmallclose} from Theorem~\ref{localnoclose}]
Let $B$ be the event that $S_{n,k}$ contains a pair of small, close components. We need a few results from~\cite{FalgasRavryWalters12a}  relating local connectivity to global connectivity.

\begin{lemma}[Lemma~2 of~\cite{FalgasRavryWalters12a}]\label{no long edges in Ubis}
For any $n$ and any integer $k$ with $0.3 \log n < k <0.6 \log n$, the probability that $U_{n,k}$ contains an edge of length at least $\frac{M\sqrt{\log n}}{8}$ is $O(n^{-6})$.
\end{lemma}
(Note that  our choice of $M$ is slightly larger than in~\cite{FalgasRavryWalters12a}; however the proof of Lemma~2 in that paper only used the fact that $M>30$, and so holds in the present setting also.)

To do away with boundary effects, we shall restrict our attention to `most' of $S_n$. Let
\[
T_n=\left[M\sqrt{\log n},\left(\left\lfloor\tfrac{\sqrt{n}}{M\sqrt{\log n}}\right \rfloor -1\right) M\sqrt{\log n}\right]^2.
\]
The nice feature of $T_n$ is that it is not very close to any of the boundary of $S_n$. The following is an easy Corollary of Theorem~1 of~\cite{Walters12}:
\begin{lemma}\label{l:Walters12}

There is a positive constant $0<c_5 <2$ such that if $k>0.3\log n$ then the probability
that $S_{n,k}$ contains any small component 
not wholly contained inside $T_n$ is $O(n^{-c_5})$.
\end{lemma}
As in~\cite{BalisterBollobasSarkarWalters09b, FalgasRavryWalters12a}, we now define two covers of $T_n$ by copies of $U_n$. The \emph{independent} cover $\mathcal{C}_1$ of $T_n$ is obtained by covering $T_n$ with copies of $U_n$ with disjoint interiors. The \emph{dominating} cover $\mathcal{C}_2$ of $T_n$ is obtained from $\mathcal{C}_1$ by replacing each square $V\in\mathcal{C}_1$ by the twenty-five translates $V+(i \frac{M\sqrt{\log n}}{4}, j \frac{M\sqrt{\log n}}{4})$, $i,j\in\{0,\pm1,\pm2\}$. By construction, we have $\mathcal{C}_1 \subseteq \mathcal{C}_2$ and the copies of $\frac{1}{4}U_n$ corresponding to elements of $\mathcal{C}_2$ cover the whole of $T_n$. Also $|\mathcal{C}_2|< \frac{25n}{M^2 \log n}$.

We shall write `$A_k$ occurs in $\mathcal{C}_i$' as a convenient shorthand for  `there is a copy $V$ of $U_n$ in $\mathcal{C}_i$ for which the event corresponding to $A_k$ occurs', and similarly for $B_k$. We shall need the following results from~\cite{FalgasRavryWalters12a}:

\begin{lemma}[Lemma~5 in~\cite{FalgasRavryWalters12a}]\label{approxbis}

For all $n \in \mathbb{N}$ and all integers $k$ with $ 0.3 \log n< k< 0.6 \log n$, and $c_5$ as given by Lemma~\ref{l:Walters12},
\[\mathbb{P}(S_{n,k} \textrm{ not connected}) = \mathbb{P}(A_k \textrm{ occurs in } \mathcal{C}_2)+O(n^{-c_5}). \]
\end{lemma}

\begin{lemma}[Lemma~6 in~\cite{FalgasRavryWalters12a}]\label{appimpbis}

There exists a constant $\gamma_1'>0$ such that for all $\varepsilon:\ 0< \varepsilon \leq \frac{1}{2}$, all integers $n> \varepsilon^{-1/\gamma_1'}$ and all integers $k$ with $k\in (0.3\log n, 0.6 \log n)$, if
\[\mathbb{P}(S_{n,k} \textrm{ connected})\geq \varepsilon\]
holds then
\[\mathbb{P}(A_k) \leq \frac{eM^2\log n}{n} \log\left(\frac{1}{\varepsilon}\right).\]
\end{lemma}

Similarly to Lemma~\ref{approxbis}, we have

\begin{lemma}\label{Bapprox}
For all $n \in \mathbb{N}$ and all integers $k$ with $ 0.3 \log n< k< 0.6 \log n$, and $c_5$ as given by Lemma~\ref{l:Walters12},
\[\mathbb{P}(B) = \mathbb{P}(B_k \textrm{ occurs in } \mathcal{C}_2)+O(n^{-c_5}). \]
\end{lemma}
\begin{proof}
This is an easy modification of the proof of Lemma~\ref{approxbis}. .
\end{proof}

Now, fix $\varepsilon:\ 0 < \varepsilon \leq \frac{1}{2}$. Suppose $\mathbb{P}(S_{n,k} \textrm{ connected})\geq \varepsilon$. Provided $n>\varepsilon^{-1/\gamma_1'}$, we have by Lemma~\ref{appimpbis} that
\[\mathbb{P}(A_k) < \frac{eM^2 \log n}{n} \log \left(\frac{1}{\varepsilon}\right).\]
Now,
\begin{align*}
\mathbb{P}(B)&=\mathbb{P}(B_k \textrm{ occurs in } \mathcal{C}_2)+O(n^{-c_5})&&\textrm{ by Lemma~\ref{Bapprox}}\\
&\leq \frac{	25 n}{M^2 \log n} \mathbb{P}(B_k) +O(n^{-c_5})&\\
& \leq  \frac{25c_3}{M^2} \frac{n^{1-c_4}}{\log n}\mathbb{P}(A_k) +O\left(\max\left(n^{-c_5},  \frac{n^{-1}}{\log n}\right)\right) &&\textrm{ by Theorem~\ref{localnoclose}}\\
& \leq 25ec_3 n^{-c_4} \log \left(\frac{1}{\varepsilon}\right)+O\left(\max\left(n^{-c_5},  \frac{n^{-1}}{\log n}\right)\right) && \textrm{ by our bound on $\mathbb{P}(A_k)$}\\
& \leq \log \left(\frac{1}{\varepsilon}\right)n^{-c_1'} &&
\end{align*}
for all $n>\varepsilon^{-1/\gamma_1'}$ and sufficiently small choices of $\gamma_1'>0$ and $c_1'>0$.
This is the claimed inequality.
\end{proof}

\section{The distribution of the small connected components}
In this section, we use Theorems~\ref{nosmallclosefinal} and~\ref{localnoclose} together  with a form of the Chen--Stein Method due to Arratia, Goldstein and Gordon~\cite{ArratiaGoldsteinGordon89} to show that the small components in $S_{n,k}$ are asymptotically Poisson distributed in a spatial as well as a numerical sense.

The Arratia, Goldstein and Gordon result essentially tells us that if the presence of one small component in a subregion of area $O(\log n)$ does not greatly increase the chance of having other small components in the same subregion, then the number of small components is Poisson distributed (just as we would expect it to be if small components were rare events occurring independently at random inside $S_n$).

We thus proceed in two stages. First of all we find a good approximation to the distribution of the small components of $S_{n,k}$ using local events. This requires us to use Theorem~\ref{nosmallclosefinal}, amongst other things. Then we adapt our local result Theorem~\ref{localnoclose} to show that the local events we define are negatively correlated --- to be more precise, they are independent if sufficiently far apart, and negatively dependent otherwise. An application of the theorem of Arratia, Goldstein and Gordon concludes the proof.

\subsection{Local approximation}
We set up some counting functions for the small components. Given a point $x \in S_n$, let $V_n(x)$ be the square of area $\left(4\lambda \sqrt{\log n}\right)^2$ centred at  $x$.

(Recall that $\lambda$ is the constant given by Definition~\ref{mconbis} in the introduction; with probability $1-o(n^{-2})$ there are no edges of length greater or equal to $\lambda\sqrt{\log n}$ and not more than one component of diameter greater than $\lambda\sqrt{\log n}$ in $S_{n,k}$.)

Also let $V_{n,k}(x)$ be the $k$-nearest neighbour graph on the set of points placed inside $V_n(x)$ by the Poisson point process on $S_n$.

\begin{definition}
Let $\Gamma$ be the \emph{grid} $\Gamma= \{x \in \mathbb{Z}^2:\ V_n(x) \subseteq S_n\}$.
Given $x \in \Gamma$, let the \emph{local counting function} $Y(x)=Y_{n,k}(x)$ be the random variable taking the value $1$ if there is a connected component $H$ in $V_{n,k}(x)$ such that $H$ has diameter less than $\lambda\sqrt{\log n}$ and $x$ is the (almost surely unique) member of $\mathbb{Z}^2$ closest to the (almost surely unique) bottom-most vertex of $H$. 
\end{definition}

Pick $x\in \Gamma$ and set 
\[p=p(n,k):=\mathbb{P}(Y(x)=1).\]
Note that $\mathbb{P}(Y(x)=1)=\mathbb{P}(Y(x')=1)$ for all $x,x'\in \Gamma$, so that the definition of $p$ is independent of $x$.

\begin{definition}
Given $x \in \Gamma$, let the \emph{global counting function} $X(x)=X_{n,k}(x)$ be the random variable taking the value $1$ if there is a connected component $H$ in $S_{n,k}$ such that $H$ has diameter less than $\lambda\sqrt{\log n}$ and $x$ is the (almost surely unique) member of $\mathbb{Z}^2$ closest to the (almost surely unique) bottom-most vertex of $H$. 
\end{definition}

We shall show that whp the small components of $S_{n,k}$ are counted exactly by $\sum_x X(x)$, and then that whp $X(x)=Y(x)$ for all $x \in \Gamma$. For this we need some easy lemmas.

\begin{lemma}\label{localitybis}
Suppose $S_{n,k}$ contains no edge of length greater than $\lambda \sqrt{\log n}$ and that $x \in \Gamma$ is such that $V_{n,k}(x)$ contains no edge of length greater than $\lambda \sqrt{\log n}$. Then $X(x)=Y(x)$.
\end{lemma}
\begin{proof} This is a minor modification of Lemma~4 of~\cite{FalgasRavryWalters12a}. 
\end{proof}

\begin{lemma}\label{globality}
Suppose $S_{n,k}$ has at most one non-small component, no two small components close together and no small component close to the boundary of $S_n$. Then there is (almost surely) a one-to-one correspondence between the small components of $S_{n,k}$ and the $x \in \Gamma$ for which $X(x)=1$.
\end{lemma}
\begin{proof}
Almost surely every connected component of $S_{n,k}$ is counted by at most one $X(x)$. Since there are no small components close to the boundary, it follows that every small component is counted by at least one $X(x)$. Finally since there are no small components close together, every $X(x)$ counts at most one component. 
\end{proof}

\begin{definition}
We set $\mathcal{D}$ to be a collection of \emph{bad events}. Let $\mathcal{D}$ be the event that any of the following occur:
\begin{enumerate}[(i)]
\item $S_{n,k}$ contains an edge of length at least $\lambda \sqrt{\log n}$
\item there is some $x \in \Gamma$ for which $V_{n,k}(x)$ contains an edge of length at least $\lambda \sqrt{\log n}$
\item $S_{n,k}$ contains at least two components of diameter at least $\lambda \sqrt{\log n}$
\item $S_{n,k}$ contains a small component $H$ such that the point of $\mathbb{Z}^2$ closest to the bottom-most vertex of $H$ does not lie in $\Gamma$
\item $S_{n,k}$ contains at least two components of diameter less than $\lambda \sqrt{\log n}$ lying within distance less than $8\lambda \sqrt{\log n}$ of each other

\item $S_{n,k}$ contains a small component $H$ such that there is more than one element of $\mathbb{Z}^2$ closest to a bottom-most vertex of $H$
\item there is some $x\in \Gamma$ for which $V_{n,k}(x)$ contains a small component $H$ such that there is more than one element of $\mathbb{Z}^2$ closest to a bottom-most vertex of $H$.
\end{enumerate}
\end{definition}

Our two previous lemmas have the following corollary:
\begin{corollary}\label{almostcount}
Suppose that $\mathcal{D}$ does not occur. Then there is a one-to-one correspondence between the small components of $S_{n,k}$ and the $x$ for which $Y(x)=1$.
\end{corollary}
\qed

How large can this bad set $\mathcal{D}$ be? All but (v) were shown whp not to occur in~\cite{FalgasRavryWalters12a} (up to some trivial changes of constants); so all we need to do is apply Theorem~\ref{nosmallclosefinal}.

\begin{lemma}\label{badset} There exists a constant $c_6>0$ such that if $k$ is an integer with $k \in (0.3\log n, 0.6 \log n)$ and $\mathbb{P}(S_{n,k} \textrm{ connected})\geq n^{-\gamma_1}$, then
\[\mathbb{P}(\mathcal{D})=o\left(n^{-c_6}\right).\]
\end{lemma}
\begin{proof}
Provided we choose $c_6$ small enough, this is immediate from the properties of $\lambda$ given in Definition~\ref{mconbis} (properties (i) and (iii)), Lemma~\ref{no long edges in Ubis} applied $n$ times (property (ii) -- note we use the fact $\lambda_2 \leq\lambda$ here), Lemma~\ref{l:Walters12} (property (iv)), Theorem~\ref{nosmallclosefinal} (property (v)) and the fact that almost surely no point of the Poisson process falls on the midpoints of two members of $\Gamma$ and no two points of the Poisson process fall on the same horizontal line (properties (vi) and (vii)). 
\end{proof}

\subsection{Global approximation}

We now study the distribution of $Y:=\sum_{x \in \Gamma}Y(x)$. Our aim is to show its law is Poisson-like. To achieve this we use the Chen--Stein Method~\cite{Chen75, Stein72} in a form due to Arratia, Goldstein and Gordon~\cite{ArratiaGoldsteinGordon89}.

Informally this says that provided that mutually dependent pairs of random variables $Y(x)$ and $Y(x')$ are not likely to both equal $1$ and that there are not too many such pairs, then the distribution of $Y$ is approximately Poisson. To state Arratia, Goldstein and Gordon's theorem precisely, we need some definitions and notation.

Let $x\in \Gamma$. We can consider a Poisson point process on $S_n$ as the union of independent Poisson point processes on $V_n(x)$ and $S_n\setminus V_n(x)$. By definition $Y(x)$ is independent of the point process on $S_n \setminus V_n(x)$.

Set $\Gamma_x$ to be the set of $y \in \Gamma$ for which $V_n(x)\cap V_n(y)\neq \emptyset$; this can be thought of as the set of possible dependencies for $Y(x)$.

Define
\begin{align*}
b_1 &= \sum_{x \in \Gamma} \sum_{y \in \Gamma_x} \mathbb{P}(Y(x) =1)\mathbb{P}(Y(y)=1)\\
b_2 &= \sum_{x \in \Gamma} \sum_{y \in \Gamma_x \setminus\{x\}} \mathbb{P}(Y(x) =1, Y(y)=1).
\end{align*}

Now let $\mu$ be the mean of $Y$,
\begin{align*}
\mu&=\mathbb{E}Y= \sum_x Y(x)\\
&=\vert\Gamma\vert p,
\end{align*}
and recall from the introduction that $\textrm{Po}_{\mu}(A)$ is the probability that a Poisson random variable with parameter $\mu$ takes a value inside the set $A$.

Then the following holds:
\begin{theorem}[Arratia, Goldstein and Gordon~\cite{ArratiaGoldsteinGordon89}]\label{AGGapprox}
\[\sup_{A \subseteq \mathbb{N}\cup\{0\}} \vert \mathbb{P}(Y \in A) - \textrm{Po}_{\mu}(A) \vert \leq b_1+b_2.\]
\end{theorem}
Thus provided we can obtain good upper bounds on $b_1$ and $b_2$, we will be close to done.

That $b_1$ is small will follow from our assumption that the probability of $S_{n,k}$ being connected is not too small. Let $c_8>0$ and $\gamma_2>0$ be strictly positive constants chosen sufficiently small to satisfy $c_8<\min(1, c_4, c_6)$ and $\gamma_2 \leq \min (\gamma_1, \frac{c_8}{2})$ respectively.

\begin{lemma}\label{appimpmod}
There is a constant $c_7>0$ such that if $\gamma \leq \min\left(\gamma_1, c_6/2\right)$, $k \in (0.3\log n, 0.6 \log n)$ and $\mathbb{P}(S_{n,k} \textrm{ connected})> n^{-\gamma}$, then
\[p < \frac{c_7 (\log n)^2}{n} \textrm{ and } \mu < c_7 (\log n)^2.\]
\end{lemma}

\begin{lemma}\label{b1small}
Suppose $k \in (0.3 \log n, 0.6 \log n)$ and
\[\mathbb{P}(S_{n,k} \textrm{ connected})> n^{-\gamma_2}.\]
Then 
\[b_1=o(n^{-c_8}).\]
\end{lemma}

\begin{proof}[Proof of Lemma~\ref{b1small} from Lemma~\ref{appimpmod}]
We have $\vert \Gamma \vert\leq 2n$.  Also, as  $Y(x)$ is independent of $Y(y)$ for all $y$ with $|| x-y||>4\sqrt{2} \lambda \sqrt{\log n}$, we have that  $\vert \Gamma_x \vert < 256 \lambda^2 \log n$.

 Now suppose $\mathbb{P}(S_{n,k} \textrm{ connected})> n^{-\gamma_2}$. We have
\begin{align*}
b_1 & =\sum_{x \in \Gamma} \sum_{y \in \Gamma_x} \mathbb{P}(Y(x) =1)\mathbb{P}(Y(y)=1)\\
&=\sum_{x\in \Gamma} \vert \Gamma_x \vert p^2\\
& \leq 2n(256 \lambda^2 \log n) p^2\\
&< 512 {c_7}^2 \lambda^2 \frac{(\log n)^5}{n}&&\textrm{by Lemma~\ref{appimpmod}}\\
&=o(n^{-c_8}) && \textrm{for  $c_8<1$}.
\end{align*}
\end{proof}

We now prove Lemma~\ref{appimpmod}.
\begin{proof}[Proof of Lemma~\ref{appimpmod}]
Suppose that $\mathbb{P}(S_{n,k} \textrm{ connected})> n^{-\gamma}$ with $\gamma \leq \min(\gamma_1, \frac{c_6}{2})$. Now
\begin{align*}
\mathbb{P}(S_{n,k} \textrm{ connected})&\leq \mathbb{P}(\{S_{n,k} \textrm{ connected}\}\cap \mathcal{D}^c)+ \mathbb{P}(\mathcal{D})\\
&=\mathbb{P}(\{Y=0\}\cap \mathcal{D}^c)+ \mathbb{P}(\mathcal{D}) &\textrm{by Corollary~\ref{almostcount}}\\
&\leq \mathbb{P}(Y=0)+\mathbb{P}(\mathcal{D}).
\end{align*}
We have $\mathbb{P}(S_{n,k} \textrm{ connected})> n^{-\gamma}$ by hypothesis and $\mathbb{P}(\mathcal{D})=o(n^{-c_6})$ by Lemma~\ref{badset}, so that
\[\mathbb{P}(Y=0)\geq n^{-\gamma} + o(n^{-c_6}).\]
Now by definition $Y(x)$ and $Y(x')$ are independent random variables for any $x, x' \in \Gamma$ with $|| x-x'||> 4\sqrt{2}\lambda\sqrt{\log n}$. Since there is a set of at least $\frac{n}{1000 \lambda^2 \log n}$ members of $\Gamma$ which are $4\sqrt{2}\lambda \sqrt{\log n}$ separated, we have
\begin{align*}
\mathbb{P}(Y=0) & \leq (1-p)^{n/1000\lambda^2\log n}.
\end{align*}
Combining this with the lower bound for $\mathbb{P}(Y=0)$ obtained above, we get
\begin{align*}
(1-p)^{n/1000\lambda^2 \log n} &\geq n^{-\gamma}+o(n^{-c_6})= n^{-\gamma}\left(1+o(n^{\gamma -c_6})\right),
\end{align*}
so that
\begin{align*}
p &\leq-\log (1-p)\\
& \leq \frac{1000\lambda^2 \gamma (\log n) ^2}{n} +o\left(\frac{\log n}{n} n^{\gamma-c_6}\right)\\
&<  \frac{c_7(\log n)^2}{n}
\end{align*}
for some suitable constant $c_7>0$ (since $\gamma \leq \frac{c_6}{2}$). In particular, $\mu =\vert \Gamma \vert p < c_7 (\log n)^2$.
\end{proof}

We now turn our attention to $b_2$.  Here we will need to use a variant of Theorem~\ref{localnoclose}:
\begin{lemma}\label{pedit}
There exists a constant $c_3'>0$ such that if $k \in (0.3\log n, 0.6\log n)$ and $x,x'$ are points in $\Gamma$ with $x' \in \Gamma_x$, then

\[\mathbb{P}(\{Y(x)=1\}\cap\{Y(x')=1\})\leq c_3'n^{-c_{4}}\mathbb{P}(\{Y(x)=1\})+ o(n^{-2}).\]
\end{lemma}

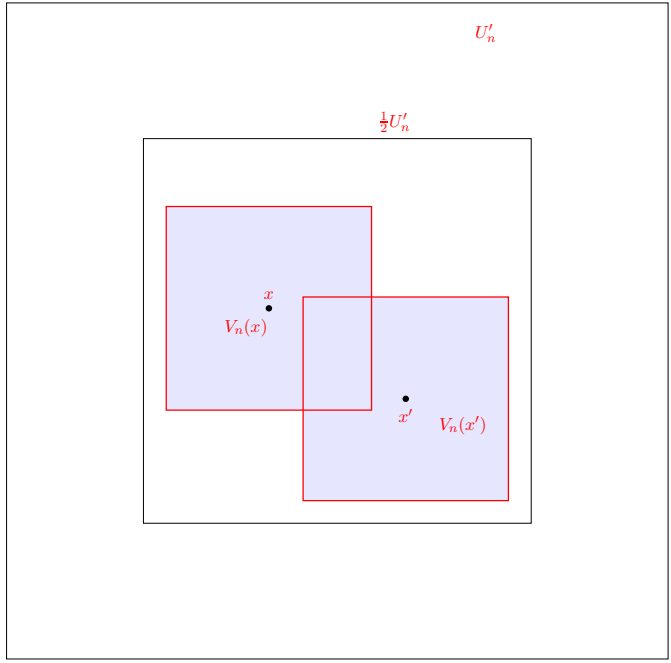
\begin{figure}
\begin{center}
\scalebox{0.6}{
\begin{tikzpicture}

\fill[blue!10] (-3, -1) rectangle (1.5,3.5);
\fill[blue!10] (0, -3) rectangle (4.5 ,1.5);

\node (x) at ( -0.75,1.25) [inner sep=0.5mm, circle, fill=black!100] {};
\node [red, above] at(-1.25, 0.5) {$V_n(x)$};
\node [red, above] at (x.north) {$x$};

\node (x') at ( 2.25,-0.75) [inner sep=0.5mm, circle, fill=black!100] {};
\node [red, below ] at (3.5, -1) {$V_n(x')$};
\node [red, below] at (x'.south) {$x'$};

\draw[thick, red!100] (-3, -1) rectangle (1.5,3.5);
\draw[thick, red!100] (0, -3) rectangle (4.5,1.5);

\draw (-3.5,-3.5) rectangle (5,5);
\node[red, above] at  (2, 5) {$\frac{1}{2}U_n'$};
\draw (-6.5, -6.5) rectangle (8, 8);
\node[red, above] at  (4, 7) {$U_n'$};

\end{tikzpicture}}
\end{center}
\caption{The intersecting squares $V_n$ and $V_n(x')$, and the translate $U_n'$ with its centre subsquare $\frac{1}{2}U_n'$ featuring in the proof of Lemma~\ref{pedit}.}

\end{figure}

\begin{proof}
Lemma~\ref{pedit} does not follow directly from Theorem~\ref{localnoclose} since instead of working with a nice square $U_n$, we begin instead with the union of two intersecting squares $V_n(x)$ and $V_n(x')$. However only a slight modification of our argument in the proof of Theorem~\ref{localnoclose} is needed to establish Lemma~\ref{pedit}.

We consider a translate $U_n'$ of $U_n$ such that the centre subsquare $\frac{1}{2}U_n'$ contains $V_n(x)\cup V_n(x')$. (This is possible since $M\geq 160 \lambda\sqrt {\log n}$ while $V_n(x), V_n(x')$ have side-length $4\lambda\sqrt{\log n}$ and $|| x-x'|| \leq 4\sqrt{2} \lambda \sqrt{\log n}$.) We consider the $k$-nearest neighbour graph $U_{n,k}'$ on the points placed inside $U_n'$ by our Poisson process on $S_n$.

We can define events $A_k(x)$ and $B_k(x,x')$ corresponding to $\{Y(x)=1\}$ and $\{Y(x)=Y(x')=1\}$ in a natural way: for $y=x,x'$ let $A_k(y)$ be the event that $U_{n,k}'$ contains a small connected component $H$ such that $y$ is the unique element of $\mathbb{Z}^2$ closest to a bottom-most point of $H$, and let $B_k(x,x')$ be the event that both of these happen, $B_k(x,x')=A_k(x)\cap A_k(x')$.

By following exactly the proof of Theorem~\ref{localnoclose}, we get
\[\mathbb{P}(B_k(x,x'))\leq c_3n^{-c_{4}}\mathbb{P}\left(A_k(x)\cup A_k(x')\right)+ o(n^{-2}).\]
Now all we have to show is that $A_k(x)$ and $A_k(x')$ are essentially the same events as $\{Y(x)=1\}$ and $\{Y(x')=1\}$ respectively. This is straightforward from Lemma~4 of~\cite{FalgasRavryWalters12a}, Lemma~\ref{localitybis} and Lemma~\ref{no long edges in Ubis}, which taken together show
\begin{align*}
\mathbb{P}(A_k(x))&=\mathbb{P}(Y(x)=1)+o(n^{-2}),\\
\mathbb{P}(A_k(x'))&=\mathbb{P}(Y(x')=1)+o(n^{-2}).
\end{align*}
Finally
\begin{align*}
\mathbb{P}\left(A_k(x)\cup A_k(x')\right)&\leq \mathbb{P}(A_k(x))+\mathbb{P}(A_k(x'))
\end{align*}
so that Lemma~\ref{pedit} follows with $c_3'=2c_3$.
\end{proof}

We are now able to bound $b_2$ from above:
\begin{lemma}\label{b2small}
Suppose $k \in (0.3 \log n, 0.6 \log n)$ and
\[\mathbb{P}(S_{n,k} \textrm{ connected})> n^{-\gamma_2}.\]
Then 
\[b_2=o(n^{-c_8}).\]
\end{lemma}
As $c_8 <\min(1, c_4)$ and $\gamma_2 \leq \min( \gamma_1, c_6/2)$, Lemma~\ref{b2small} comes as a straightforward consequence of Lemma~\ref{pedit}.
\begin{proof}
Suppose that $\mathbb{P}(S_{n,k} \textrm{ connected})> n^{-\gamma}$ with $\gamma \leq \min(\gamma_1, \frac{c_6}{2})$. Then by Lemma~\ref{appimpmod}
\[p<c_7 \frac{(\log n)^2}{n}.\]

Now $\vert \Gamma \vert <2n$ and (as $Y(x)$ is independent of $Y(x')$ for all $x'$ with $|| x- x'||\geq 4\sqrt{2}\lambda \sqrt{\log n}$),  $\vert \Gamma_x \vert <256 \lambda^2 \log n$ for all $x \in \Gamma$.

Thus
\begin{align*}
b_2 &= \sum_{x \in \Gamma} \sum_{y \in \Gamma_x \setminus\{x\}} \mathbb{P}(Y(x) =1, Y(y)=1)\\
& \leq \sum_{x \in \Gamma} \vert \Gamma_x \vert \left(c_3'n^{-c_4}p +o(n^{-2})\right)&&\textrm{by Lemma~\ref{pedit}}\\
&\leq 512\lambda^2c_3' (\log n) n^{1-c_4'}p +o\left(\frac{\log n}{n}\right)&&\textrm{by our bounds on $\vert \Gamma\vert, \vert \Gamma_x \vert$} \\
&< 512\lambda^2c_3'c_7 \frac{(\log n)^3}{n^{c_4}} + o\left(\frac{\log n}{n}\right) && \textrm{by our bound on $p$}\\
&= o(n^{-c_8}) && \textrm{for $c_8 <\min(1, c_4)$.}
\end{align*}
\end{proof}

\begin{corollary}\label{yispoisson}
Suppose $k \in (0.3 \log n, 0.6 \log n)$ and $\mathbb{P}(S_{n,k} \textrm{ connected})>n^{-\gamma_2}$. Then
\[\sup_{A \subseteq \mathbb{N}\cup\{0\}} \vert \mathbb{P}(Y \in A) - \textrm{Po}_{\mu}(A) \vert = o(n^{-c_8}).\]
\end{corollary}
\begin{proof}
The hypotheses of Lemma~\ref{b1small} and Lemma~\ref{b2small} are satisfied, so that $b_1+b_2=o(n^{-c_8})$. Thus by Theorem~\ref{AGGapprox},
\begin{align*}
\sup_{A \subseteq \mathbb{N}\cup\{0\}} \vert \mathbb{P}(Y \in A) - \textrm{Po}_{\mu}(A) \vert& \leq b_1 +b_2=o(n^{-c_8}).
\end{align*}
\end{proof}

\subsection{Putting it together}
We now put together the results of the previous sections to prove Theorem~\ref{poissondistrfinal}.

Corollary~\ref{yispoisson} tells us that $Y$ is approximately Poisson with parameter $\mu=\mathbb{E}Y$, while Corollary~\ref{almostcount} tells us that whp $Y$ counts exactly the small components of $S_{n,k}$. Our aim is to show that the number of small components $X$ of $S_{n,k}$ is approximately Poisson with parameter $\nu=\nu(n,k)$, where
\[\nu(n,k):=-\log \mathbb{P}(S_{n,k} \textrm{ connected}).\]

Thus what we have left to show is that $\textrm{Po}_{\mu}$ and $\textrm{Po}_{\nu}$ are approximately the same probability measure. We do this in two stages: first we prove $\mu$ and $\nu$ are almost equal, and then use that to show $\textrm{Po}_{\mu}$ and $\textrm{Po}_{\nu}$ are approximately the same.

Let $c_8>0$ and $\gamma_2>0$ be the constants defined in the previous subsection. Recall that they satisfy $c_8<\min(1, c_4, c_6)$ and $\gamma_2 \leq \min (\gamma_1, \frac{c_8}{2})$ respectively. 

\begin{lemma}\label{munuaresame}
Suppose $k\in (0.3 \log n, 0.6 \log n)$ and $\mathbb{P}(S_{n,k} \textrm{ connected})> n^{-\gamma_2}$. Then,
\[\nu=\mu+o(n^{-c_8/2}).\]
\end{lemma}
\begin{proof}
We have
\begin{align*}
e^{-\nu}&=\mathbb{P}(S_{n,k} \textrm{ connected}) &&\\
& = \mathbb{P}\left(\{S_{n,k} \textrm{ connected}\} \cap \mathcal{D}^c\right)+ \mathbb{P}\left(\{S_{n,k} \textrm{ connected}\}\cap \mathcal{D}\right)&&\\
& = \mathbb{P}\left( \{Y= 0\} \cap \mathcal{D}^c\right)+\mathbb{P}\left(\{S_{n,k} \textrm{ connected}\} \cap \mathcal{D}\right) &&\textrm{by Corollary~\ref{almostcount}}\\
&= \mathbb{P}(Y=0) +o\left( n^{-c_6}\right) &&\textrm{by Lemma~\ref{badset}}\\
&= e^{-\mu} +o\left(\max(n^{-c_8}, n^{-c_6})\right) &&\textrm{by Corollary~\ref{yispoisson}}\\
&=e^{-\mu}+o(n^{-c_8})&& \textrm{since $c_8<c_6$.}
\end{align*}
Now $e^{-\nu}> n^{-\gamma_2}$ and $\gamma_2 < \frac{c_8}{2}$ by assumption, so that \begin{align*}
\mu &= - \log \left(e^{-\nu}+o\left(n^{-c_8}\right)\right)\\
&= \nu -\log \left( 1+o(n^{-(c_8-\gamma_2)})\right)\\
&= \nu + o\left(n^{-(c_8-\gamma_2)}\right)\\
&= \nu +o(n^{-c_8/2})
\end{align*}
as claimed.
\end{proof}

It readily follows that $\textrm{Po}_{\mu}$ and $\textrm{Po}_{\nu}$ are close to being the same measure. 
\begin{corollary}\label{munu}
If $k \in (0.3\log n, 0.6 \log n)$ and
\[\mathbb{P}(S_{n,k} \textrm{ connected})>n^{-\gamma_2}\]
both hold, then
\[ \sup_{A \subset \mathbb{N}\cup\{0\}} \left\vert \textrm{Po}_{\nu}(A)- \textrm{Po}_{\mu}(A)\right\vert = o\left(n^{-c_{8}/2}\right).\]
\end{corollary}
\begin{proof}
Assume without loss of generality that $\nu \geq \mu$. Then we can consider a Poisson random variable with mean $\nu$ as the sum of two independent Poisson random variables with means $\mu$ and $\nu-\mu$ respectively. Thus for any set $A \subseteq \mathbb{N}\cup\{0\}$,
\begin{align*}
\textrm{Po}_{\nu}(A)&\geq \textrm{Po}_{\mu}(A)\textrm{Po}_{\nu-\mu}(0)\\
&= \textrm{Po}_{\mu}(A)\exp(-(\nu-\mu))\\
&= \textrm{Po}_{\mu}(A) +o(n^{-c_8/2}) && \textrm{by Lemma~\ref{munuaresame}.}
\end{align*}
In the other direction,
\begin{align*}
\textrm{Po}_{\nu}(A)&\leq \textrm{Po}_{\mu}(A)\textrm{Po}_{\nu-\mu}(0)+ \textrm{Po}_{\nu-\mu}\left(\left[1, \infty\right)\right)\\
&= \textrm{Po}_{\mu}(A)\exp(-(\nu-\mu))+ (1-\exp(-(\nu-\mu))\\
&= \textrm{Po}_{\mu}(A) +o(n^{-c_8/2}) && \textrm{by Lemma~\ref{munuaresame}.}
\end{align*}
Thus for all $A\subseteq \mathbb{N}\cup \{0\}$ we have
\[\vert \textrm{Po}_{\mu}(A)- \textrm{Po}_{\nu}(A) \vert = o(n^{-c_8/2}),\]
as claimed.

\end{proof}

Theorem~\ref{poissondistrfinal} then follows from the triangle inequality and appropriately small choices of the constants $c_2>0$ and $\gamma_2>0$.

\begin{proof}[Proof of Theorem~\ref{poissondistrfinal}]
Let $X=X_{n,k}$ denote the number of small components in $S_{n,k}$. Assume $\mathbb{P}(S_{n,k} \textrm{ connected})>n^{-\gamma_2}$.

The proof of Theorem~5 of~\cite{BalisterBollobasSarkarWalters05} establishes the existence of a constant $\gamma>0$ such that for $k\leq 0.3 \log n$ we have
\[\mathbb{P}(S_{n,k} \textrm{ connected})=o(n^{-\gamma}).\]
Provided $\gamma_2 < \gamma$ and $n$ is sufficiently large this together with our assumption on the connectivity of $S_{n,k}$ guarantees $k>0.3\log n$.

Also Theorem~13 of~\cite{BalisterBollobasSarkarWalters05} shows that for $k \geq 0.6 \log n$ the probability that $S_{n,k}$ contains \emph{any} small component is $o\left(n^{-(0.6\log 7 -1)}\right)$. Thus if $c_2< 0.6 \log 7- 1=0.16\ldots$ and $k \geq 0.6 \log n$ we have
\[\mathbb{P}(X=0)=1-o(n^{-c_2})\]
and
\[\nu= - \log \left(\mathbb{P}(S_{n,k} \textrm{ connected})\right)=o(n^{-c_2})\]
so that
\[\textrm{Po}_{\nu}(\{0\})=1-o(n^{-c_2})\]
and
\[\textrm{Po}_{\nu}((1, \infty))=o(n^{-c_2}).\]
Thus in this case
\begin{align*}
\sup_{A \in \mathbb{N} \cup\{0\} } \vert \mathbb{P}(X \in A) - \textrm{Po}_{\nu}(A)\vert
& \leq \vert \mathbb{P}(X=0)-\textrm{Po}_{\nu}(\{0\})\vert+ \mathbb{P}(X\neq 0)+\textrm{Po}_{\nu}((1, \infty))\\
&=o(n^{-c_2}),
\end{align*}
so that the conclusion of Theorem~\ref{poissondistrfinal} holds in this case.

Now supppose $k \in (0.3\log n, 0.6 \log n)$. Then,
\begin{align*}
\sup_{A \in \mathbb{N} \cup\{0\} }& \vert \mathbb{P}(X \in A) - \textrm{Po}_{\nu}(A)\vert\\
&\leq \sup_{A \in \mathbb{N}\cup \{0\}} \vert\mathbb{P}(X \in A)- \mathbb{P}(Y \in A)\vert+\sup_{A \in \mathbb{N}\cup \{0\}} \vert\mathbb{P}(Y \in A)- \textrm{Po}_{\mu}(A)\vert \\
& \qquad \qquad +\sup_{A \in \mathbb{N}\cup \{0\}}  \vert\textrm{Po}_{\mu}(A)- \textrm{Po}_{\nu}(A)\vert\\
&\leq 2\mathbb{P}(\mathcal{D})+o(n^{-c_8/2})\quad \textrm{by Corollary~\ref{almostcount}, Corollary~\ref{yispoisson} and Corollary~\ref{munu}}\\
&=o(n^{-c_2}) \quad \textrm{by Lemma~\ref{badset}, provided $c_2<\min(c_6, c_8/2)$.}
\end{align*}

Thus picking $c_2$ and $\gamma_2$ to satisfy
\[0 < c_2 <  \min\left(c_6,0.6 \log 7 -1, \frac{c_8}{2}\right)\]
and
\[0<\gamma_2 < \min \left(\gamma, \gamma_1, \frac{c_8}{2}\right)\]
we are done. 

\end{proof}

\subsection{Process version}

Let us conclude this paper by showing as promised that the locations of the small components are approximately distributed like a Poisson process.

Let $\mathbf{X}$ and $\mathbf{Y}$ be the $\vert \Gamma \vert $-dimensional vectors $\mathbf{X}=(X(x))_{x \in \Gamma}$ and $\mathbf{Y}=(Y(x))_{x \in \Gamma}$ respectively. We define two Poisson processes on $\Gamma$.

Let $(Z(x))_{x \in \Gamma}$ be a set of $\vert \Gamma \vert$ independent, identically distributed random variables, with $Z(x) \sim \rm{Poisson}(p)$ for every $x$. (Recall $p=p(n,k)=\mathbb{P}(Y(x)=1)$.) Set $\mathbf{Z}$ to be the resulting Poisson process on $\Gamma$, $\mathbf{Z}=(Z(x))_{x \in \Gamma}$.

Similary set $p'=\nu/\vert \Gamma \vert$ and let $(Z'(x))_{x \in \Gamma}$ be a set of $\vert \Gamma \vert$ independent, identically distributed random variables, with $Z'(x) \sim \rm{Poisson}(p')$ for every $x$. Set $\mathbf{Z'}$ to be the resulting Poisson process on $\Gamma$, $\mathbf{Z'}=(Z'(x))_{x \in \Gamma}$.

Arratia, Goldstein and Gordon~\cite{ArratiaGoldsteinGordon89} gave the following process version of their Poisson approximation theorem: 
\begin{theorem}[Arratia, Goldstein and Gordon~\cite{ArratiaGoldsteinGordon89}]\label{AGGspatial}
Let $b_1$ and $b_2$ be as in Theorem~\ref{AGGapprox}. Then
\[\sup\left\{\left\vert \mathbb{P}(\mathbf{Y}\in \mathcal{A})-\mathbb{P}(\mathbf{Z}\in\mathcal{A})\right\vert : \ \mathcal{A} \subseteq \left(\mathbb{N}\cup\{0\}\right)^{\vert \Gamma \vert} \right\} \leq 2(b_1+b_2).\]
\end{theorem}
We shall write $D(\mathbf{Y}, \mathbf{Z})$ for the difference 
\[D(\mathbf{Y}, \mathbf{Z}):=\sup\left\{\left\vert \mathbb{P}(\mathbf{Y}\in\mathcal{A})-\mathbb{P}(\mathbf{Z}\in\mathcal{A})\right\vert : \ \mathcal{A} \subseteq \left(\mathbb{N}\cup\{0\}\right)^{\vert \Gamma \vert} \right\},\]
which is known as the \emph{total variation distance} 
between the distributions of $\mathbf{Y}$ and $\mathbf{Z}$.

Let us now apply Theorem~\ref{AGGspatial} to prove a process version of Theorem~\ref{poissondistrfinal}:
\begin{theorem}\label{poissondistrspace}
Suppose 
$\mathbb{P}(S_{n,k} \textrm{ connected})> n^{-\gamma_2}.$
Then
\[D(\mathbf{X}, \mathbf{Z'}) = o(n^{-c_2}).\]
\end{theorem}

\begin{proof}
Suppose $\mathbb{P}(S_{n,k} \textrm{ connected})> n^{-\gamma_2}$. As in the proof of Theorem~\ref{poissondistrfinal}, we may assume $0.3\log n<k<0.6\log n$ provided $n$ is sufficiently large.

We know by Corollary~\ref{almostcount} that if $\mathcal{D}^c$ does not occur then $\mathbf{X}=\mathbf{Y}$. Theorem~\ref{AGGspatial} and the bounds on $b_1$ and $b_2$ we obtained in Lemmas~\ref{b1small} and~\ref{b2small} tell us $D(\mathbf{Y}, \mathbf{Z})$ is at most $o(n^{-c_2})$. We know by Lemma~\ref{munuaresame} that $\nu =\mu +o(n^{-c_8/2})$. Dividing through by $\vert \Gamma \vert$ we get that 
\begin{align*}
p'=p +o(n^{-(1+c_8/2)}),\tag{5}
\end{align*} from which it is easy to show that $D(\mathbf{Z}, \mathbf{Z}')=o(n^{-c_2})$:
running exactly the same argument as in Corollary~\ref{munu} but with $p$ and $p'$ instead of $\mu$ and $\nu$, and using (5) instead of Lemma~\ref{munuaresame}, we get that
\[\sup_{A \subset \mathbb{N}\cup\{0\}} \vert \mathbb{P}(Z(x) \in A)- \mathbb{P}(Z'(x) \in A)\vert= o(n^{-1-c_8/2})\]
for all $x \in \Gamma$. Thus
\begin{align*}
D(\mathbf{Z}, \mathbf{Z}')& \leq \sum_{x \in \Gamma} \sup_{A \subset \mathbb{N}\cup\{0\}} \vert \mathbb{P}(Z(x) \in A)- \mathbb{P}(Z'(x) \in A)\vert\\
& =o(\vert \Gamma \vert n^{-1-c_8/2})\\
&=o(n^{-c_2}),
\end{align*}
as required.

Putting it all together, we have
\begin{align*}
D(\mathbf{X}, \mathbf{Z'})& \leq D(\mathbf{X}, \mathbf{Y})+D(\mathbf{Y}, \mathbf{Z})+D(\mathbf{Z}, \mathbf{Z'})\\
&\leq 2 \mathbb{P}(\mathcal{D})+o(n^{-c_2})\\
&=o(n^{-c_2}) &&\textrm{by Lemma~\ref{badset}, since $c_2<c_6$.}
\end{align*}
\end{proof}

What Theorem~\ref{poissondistrspace} effectively says is that the location of the small components of $S_{n,k}$ inside $S_n$ is approximately a Poisson point process. Thus the distribution of the small components is approximately Poisson in both a numerical and in a spatial sense.

\section{Acknowledgements}
I am grateful to Paul Balister, Oliver Riordan and Mark Walters for helpful comments and discussions, which helped greatly improve the presentation of this article.

I would also like to thank the anonymous referee for his/her careful work on the paper.
\bibliographystyle{plainnat}
\bibliography{masterbibliography}

\begin{thebibliography}{14}
\providecommand{\natexlab}[1]{#1}
\providecommand{\url}[1]{\texttt{#1}}
\expandafter\ifx\csname urlstyle\endcsname\relax
  \providecommand{\doi}[1]{doi: #1}\else
  \providecommand{\doi}{doi: \begingroup \urlstyle{rm}\Url}\fi

\bibitem[Arratia et~al.(1989)Arratia, Goldstein, and
  Gordon]{ArratiaGoldsteinGordon89}
R.~Arratia, L.~Goldstein, and L.~Gordon.
\newblock Two moments suffice for {P}oisson approximations: the {C}hen--{S}tein
  method.
\newblock \emph{The Annals of Probability}, 17\penalty0 (1):\penalty0 9--25,
  1989.

\bibitem[Balister and Bollob{\'a}s(2009)]{BalisterBollobas12}
P.~Balister and B.~Bollob{\'a}s.
\newblock Percolation in the k-nearest neighbor graph, 2009.

\bibitem[Balister et~al.(2005)Balister, Bollob{\'a}s, Sarkar, and
  Walters]{BalisterBollobasSarkarWalters05}
P.~Balister, B.~Bollob{\'a}s, A.~Sarkar, and M.~Walters.
\newblock Connectivity of random k-nearest-neighbour graphs.
\newblock \emph{Advances in Applied Probability}, 37\penalty0 (1):\penalty0
  1--24, 2005.

\bibitem[Balister et~al.(2009{\natexlab{a}})Balister, Bollob{\'a}s, Sarkar, and
  Walters]{BalisterBollobasSarkarWalters09b}
P.~Balister, B.~Bollob{\'a}s, A.~Sarkar, and M.~Walters.
\newblock A critical constant for the k nearest-neighbour model.
\newblock \emph{Advances in Applied Probability}, 41\penalty0 (1):\penalty0
  001--012, 2009{\natexlab{a}}.

\bibitem[Balister et~al.(2009{\natexlab{b}})Balister, Bollob{\'a}s, Sarkar, and
  Walters]{BalisterbollobasSarkarWalters09a}
P.~Balister, B.~Bollob{\'a}s, A.~Sarkar, and M.~Walters.
\newblock Highly connected random geometric graphs.
\newblock \emph{Discrete Applied Mathematics}, 157\penalty0 (2):\penalty0
  309--320, 2009{\natexlab{b}}.

\bibitem[Chen(1975)]{Chen75}
L.H.Y. Chen.
\newblock Poisson approximation for dependent trials.
\newblock \emph{The Annals of Probability}, pages 534--545, 1975.

\bibitem[Falgas-Ravry and Walters(2012)]{FalgasRavryWalters12a}
V.~Falgas-Ravry and M.~Walters.
\newblock Sharpness in the k-nearest neighbour random geometric graph model.
\newblock \emph{Advances in Applied Probability}, 44:\penalty0 617--634, 2012.

\bibitem[Gilbert(1961)]{Gilbert61}
E.N. Gilbert.
\newblock Random plane networks.
\newblock \emph{Journal of the Society for Industrial and Applied Mathematics},
  9\penalty0 (4):\penalty0 533--543, 1961.

\bibitem[Gonz{\'a}lez-Barrios and Quiroz(2003)]{GonzalesBarriosQuiroz03}
J.M. Gonz{\'a}lez-Barrios and A.J. Quiroz.
\newblock A clustering procedure based on the comparison between the k-nearest
  neighbors graph and the minimal spanning tree.
\newblock \emph{Statistics and probability letters}, 62\penalty0 (1):\penalty0
  23--34, 2003.

\bibitem[Penrose(1997)]{Penrose97}
M.D. Penrose.
\newblock The longest edge of the random minimal spanning tree.
\newblock \emph{The annals of applied probability}, 7\penalty0 (2):\penalty0
  340--361, 1997.

\bibitem[Stein(1972)]{Stein72}
C.~Stein.
\newblock A bound for the error in the normal approximation to the distribution
  of a sum of dependent random variables.
\newblock In \emph{Proc. Sixth Berkeley Symp. Math. Stat. Prob.}, pages
  583--602, 1972.

\bibitem[Walters(2011)]{Walters11}
M.~Walters.
\newblock Random geometric graphs.
\newblock In \emph{Surveys in Combinatorics 2011}, pages 365--402. Cambridge
  University Press, 2011.

\bibitem[Walters(2012)]{Walters12}
M.~Walters.
\newblock Small components in k-nearest neighbour graphs.
\newblock \emph{Discrete Applied Mathematics}, 160:\penalty0 2037--2047, 2012.

\bibitem[Xue and Kumar(2004)]{XueKumar04}
F.~Xue and P.R. Kumar.
\newblock The number of neighbors needed for connectivity of wireless networks.
\newblock \emph{Wireless networks}, 10\penalty0 (2):\penalty0 169--181, 2004.

\end{thebibliography}

\end{document}